\documentclass{article}

\usepackage{amsthm}
\usepackage{graphicx}
\usepackage{amssymb}
 \usepackage{mathptmx}      

\newtheorem{theorem}{Theorem}
\newtheorem{corollary}{Corollary}[theorem]

\theoremstyle{definition}
\newtheorem{example}{Example}[section]

\newtheorem*{remark}{Remark}

\begin{document}

\title{On a Dirichlet to Neumann and Robin to Neumann  operators suitable for reflecting harmonic functions subject to a nonhomogeneous condition on an arc
}


\author{Murdhy Aldawsari      
  \and
        Tatiana Savina 
}




\maketitle

\begin{abstract}
According to the  Schwarz symmetry principle, every harmonic function vanishing on a real analytic curve has an odd continuation, while a harmonic function satisfying homogeneous Neumann condition has the even continuation. 
There are different generalizations of the Schwarz symmetry principle. Most of them are dealing with homogeneous conditions and the Dirichlet case. Using a technique of Dirichlet to Neumann and Robin to Neumann operators, we derive reflection formulae for  nonhomogeneous Neumann and Robin conditions from a reflection formula subject to a nonhomogeneous Dirichlet condition.  

{\bf Keywords: }{Schwarz symmetry principle,  Dirichlet to Neumann operator, Robin to Neumann  operator, Analytic continuation}
\end{abstract}


%
%


\section{Introduction}
\label{intro}

Let $\Gamma \subset \mathbb{R}^2$ be a non-singular real analytic
curve and a point $z_0=(x_0,y_0)\in \Gamma$. Then, there exists a neighborhood $\Omega$ of
$z_0$ and an anti-conformal mapping $R:\Omega\to \Omega$ which is identity on
$\Gamma$, permutes the components $\Omega _1, \Omega _2$ of $\Omega\setminus \Gamma$
and relative to which any harmonic function $u(x, y)$ defined near
$\Gamma $ and vanishing on $\Gamma $ (the homogeneous Dirichlet condition) is
odd,
\begin{equation} \label{E:0.1}
         u(x ,y ) =-u(R(x ,y)).
\end{equation}

In the case of nonhomogeneous Dirichlet data, $u=\varphi (x,y)$ on $\Gamma$, when function $\varphi$ is holomorphically continuable into $\mathbb{C}^2$ near $\Gamma$, formula (\ref{E:0.1}) involves also values of function $u$ at two additional  points located on the  complexification  $\Gamma _{\mathbb{C}}$ of the curve $\Gamma$. All four points then create a so-called Study's rectangle \cite{226}.

To describe the Study's rectangle, consider a complex domain $W$ in the space $\mathbb{C}^2$ to which
the function $f$ defining the curve $\Gamma :=\{f(x,y)=0\}$  can be
 analytically continued such that $W\cap \mathbb{R}^2=U$. Using the
change of variables $z=x+iy$, $\zeta=x-iy,$ the equation of the
complexified curve $\Gamma _{\mathbb{C}}$ can be rewritten in the
form
$$
f\left ( \frac{z+\zeta}{2}, \, \frac{z-\zeta}{2i} \right ) =0,
$$
and if $grad \, f(x,y)\ne 0$ on $\Gamma$, can be also rewritten in
terms of the Schwarz function and its inverse, $w=S(z)$ and
$z=\widetilde{S}(\zeta)$ \cite{davis}. The mapping
$R$ mentioned above can be expressed in terms of the Schwarz function as follows,
\begin{equation} \label{E:0.2}
R(x ,y )=R(z )=\overline{S(z )}.
\end{equation}
Using the above notations, the reflection formula for harmonic functions subject to conditions 
$u_|{_{\Gamma}}=\phi$ can be written as the Study's rectangle:
\begin{equation} \label{E:0.1n}
 u(z ,\zeta ) + u(\widetilde{S}(\zeta) ,S(z ))=\varphi (\widetilde{S}(\zeta),\zeta)+\varphi (z, S(z)).
\end{equation}
The motivation of this paper goes back to D.~Khavinson's suggestion to think of a different method for deriving  reflection formulae rather than using the reflected fundamental solution method, described, for example, in \cite{nonloc}. This suggestion perhaps was rather related  to methods used in analysis. However, L.~Beznea's talk, devoted to  Dirichlet to Neumann operators, given at the International Conference on Complex Analysis, Potential Theory and Applications in honour of Professor Stephen J Gardiner on the occasion of his 60th Birthday, influenced this attempt to derive reflection formulas for other than Dirichlet conditions using similar to Dirichlet to Neumann operators studied by L.~Beznea et al. \cite{beznea}.

The structure of the paper is as follows. In section \ref{prelim} we discuss some preliminaries. In section \ref{R2} we discuss a Dirichlet to Neumann operator for a circular arc written in a form suitable for reflection. In section \ref{R3} we use this formula to derive a reflection formula for the case of nonhomogeneous Neumann conditions given on a circular arc. In section \ref{R5} we discuss a Robin to Neumann operator for a circular arc. Section \ref{R6} is devoted to reflection for the case of a nonhomogeneous Robin data on a circular arc, and section \ref{R7} discusses a Dirichlet to Neumann operator and reflection for a harmonic function subject to nonhomogeneous Neumann condition on an arc of an algebraic curve.

\section{Preliminaries}
\label{prelim}

Consider the Neumann boundary value problem in the unit disk $\mathbb{D}$ in the plane:
\begin{equation}
\left\{
                \begin{array}{ll}\label{N}
                  \Delta V=0 \qquad \mbox{in} \quad\mathbb{D}\\
                  \frac{\partial V}{\partial n}=\varphi           \qquad \mbox{on} \quad \partial\mathbb{D},\\
                 \end{array}
              \right.
\end{equation}
where $n$ is the outward normal.
Solution to this problem can be written in terms of the solution to a corresponding Dirichlet problem, $V=\Lambda (U)$  \cite{beznea}.  Here $U(x,y)$ is defined by
\begin{equation}
\left\{
                \begin{array}{ll}\label{D}
                  \Delta U=0 \qquad \mbox{in} \quad\mathbb{D}\\
                  U=\varphi           \qquad \mbox{on} \quad \partial\mathbb{D},\\
                 \end{array}
              \right.
\end{equation}
where $\varphi \in {\Large C}(\mathbb{\partial D})$.

\begin{theorem}\cite{beznea} Assume $\varphi : \partial\mathbb{D} \to \mathbb{R}$ is continuous and its integral along $\partial\mathbb{D}$ vanishes. If $u$ is the solution to the Dirichlet problem (\ref{D}), then 
\begin{equation}\label{L0}
V(z)=\Lambda (U)= \int\limits _0^ 1\frac{U(\rho z)}{\rho} \, d(\rho ), \quad z\in 
\mathbb{D}\cup \partial\mathbb{D},
\end{equation}
is the solution to the Neumann problem (\ref{N}) with $V(0)=0$, where $z=x+iy$.
\end{theorem}
A derivation of  formula (\ref{L0}) in paper \cite{beznea} was based on the following consideration. If $U(x,y)=\Re [g(z)]$ and $V(x,y)=\Re [f(z)]$, where  $g(z)$ and $f(z)$ are analytic functions,  then the following equalities hold on $\partial\mathbb{D}$,
$$
\varphi = U = \frac{\partial V}{\partial n}= (\nabla\, V\cdot n)=
\frac{\partial V}{\partial x}x+\frac{\partial V}{\partial y}y=\Re [zf^{\prime}(z)]
=\Re[g(z)].
$$
Setting $zf^{\prime}(z)=g(z)$ and integrating with from $z$ to $z_0$, one has
\begin{equation}\label{L1}
f(z)=f(z_0)- \int\limits _z^ {z_0}\frac{g(\tau)}{\tau} \, d\tau , \quad z, z_0\in 
\mathbb{D}\cup \partial\mathbb{D} .
\end{equation}
Formula (\ref{L1}) is a starting point for writing an expression relating harmonic functions
satisfying Dirichlet, Neumann, and Robin conditions on $\partial\mathbb{D}$.

\section{Dirichlet to Neumann operator for an arc of the unit circle}
\label{R2}

Consider the following two problems for functions $u(x,y)$ and $v(x,y)$ defined near an arc 
$\partial\mathbb{D}$ of a unit circle in the plane:
\begin{equation}
\left\{
                \begin{array}{ll}\label{D1}
                  \Delta u=0 \qquad \mbox{near} \quad\partial\mathbb{D}\\
                  u=\varphi           \qquad \mbox{on} \quad \partial\mathbb{D},\\
                 \end{array}
              \right.
\end{equation}
\begin{equation}
\left\{
                \begin{array}{ll}\label{N1}
                  \Delta v=0 \qquad \mbox{near} \quad\partial\mathbb{D}\\
                  \frac{\partial v}{\partial n}=\varphi           \qquad \mbox{on} \quad \partial\mathbb{D},\\
                 \end{array}
              \right.
\end{equation}
where $\varphi$ is  holomorphically continuable into $\mathbb{C}^2$ near $\partial\mathbb{D}$ .
 To find and expression relating $v(x,y)$ and 
$u(x,y)$ one can take the real part of  (\ref{L1}) by computing $v(x,y)=
\frac{1}{2}(f(z)+\overline{f(z)})$.
However, for what follows it is more convenient to complexify the problem, 
considering functions $v(z,\zeta )$ and $u(z,\zeta )$, where $(z,\zeta )\in \mathbb{C}^2$ with $\zeta =\bar z$ on
$\mathbb{R}^2$.
\begin{theorem}
Let $\partial\mathbb{D}$ be the unit circle in the plane centered at the origin. 
If $u(z,\zeta)$ is a solution to (\ref{D1}), then 
\begin{equation}\label{DN}
v(z,\zeta)=v(z_0,\zeta _0)
- \int\limits _z ^{z_0} \frac{u_1(\tau )}{\tau}d\tau
-\int\limits _\zeta ^{\zeta _0} \frac{u_2(\xi )}{\xi}d\xi 
\end{equation}
is a solution to (\ref{N1}). 
Here $u(z,\zeta)=u_1(z)+u_2(\zeta )$ is a harmonic function such that 
$u(z,\bar z )=\varphi$ on $\partial\mathbb{D}$.
\end{theorem}
\begin{proof}
To prove  formula (\ref{DN}) one just needs to check  the conditions in (\ref{N1}).
Obviously, function $v(z,\zeta ) $ is harmonic, $\frac{\partial ^2 v}{\partial z\partial \zeta}=0$.


To check the  second condition in (\ref{N1}), we fix a point $(x_0,y_0) \in \mathbb{\partial D}$.
Using the notation $z_0=x_0+iy_0$, the outward normal at $z_0$, $n (z_0)=z_0$, can be computed as
$$\frac{\partial v}{\partial n}(z_0)
=\lim_{h\to 0}\frac{v(z_0+h \, z_0)-v(z_0)}{h}
=\lim_{h\to0}\frac{v(z_0(h+1))-v(z_0)}{h}.$$
Taking into account the expression (\ref{DN}), we have
$$\frac{\partial v}{\partial n}(z_0)
=\lim_{h\to0}\frac{1}{h} ( \int_{e^{i\theta}}^{(h+1)e^{i\theta}}\frac{u_1(\tau)}{\tau} d\tau
+\int_{e^{-i\theta}}^{(h+1)e^{-i\theta}}\frac{u_2(\xi)}{\xi} d\xi )
.$$
Using the substitutions $\tau=r e^{i\theta}$, $\xi=r e^{-i\theta}$ with fixed $\theta$, the above formula yields
$$
\frac{\partial v}{\partial n}(z_0)=
\lim_{h\to0}\frac{1}{h}{\int_{1}^{h+1}\frac{u_1(r e^{i\theta})+u_2(r e^{-i\theta})}{r} d r}=
\lim_{\widetilde{r}\to 1}\frac{u(\widetilde{r}e^{i\theta},\widetilde{r}e^{-i\theta})}{\widetilde{r}}=
u(e^{i\theta},e^{-i\theta})=\varphi(x_0,y_0).$$
This finishes the proof.
\end{proof}

\begin{example}
Let harmonic function satisfying the Dirichlet condition $\varphi =C$ be $u(x,y)=C$, where $C$ is a constant. Then for function $v(x,y)$, we have,
$$
v(z,\zeta)=v(z_0,\zeta _0)
- \frac{1}{2}\int\limits _z ^{z_0} \frac{C}{\tau}d\tau
-\frac{1}{2}\int\limits _\zeta ^{\zeta _0} \frac{C}{\xi}d\xi ,
$$
which implies that the corresponding to $u(x,y)=C$ function, satisfying Neumann condition 
$\frac{\partial v}{\partial n}=C$, is $v(x,y)={C}\ln \sqrt{x^2+y^2}+ const$.
\end{example}
\begin{example}
Let harmonic function satisfying homogeneous Dirichlet condition be 
$u(x,y)=\ln \sqrt{x^2+y^2}=\frac{1}{2}\ln z +\frac{1}{2}\ln \zeta$. Let us compute the corresponding function $v(x,y)$ with $\varphi =0$,
$$
v(z,\zeta)=v(z_0,\zeta _0)
- \frac{1}{2}\int\limits _z ^{z_0} \frac{\ln \tau}{\tau}d\tau
-\frac{1}{2}\int\limits _\zeta ^{\zeta _0} \frac{\ln \xi}{\xi}d\xi ,
$$
which implies that the corresponding solution of problem (\ref{N1}) is
$$
v(x,y)=\frac{1}{2}(\ln ^2 \, \sqrt{x^2+y^2}-\arctan ^2 \, \frac{y}{x})+ const.
$$
The latter expression in polar coordinates, $z=re^{i\theta}$, has a simpler form,
$v(r, \theta )= \frac{1}{4}(\ln ^2 \, r-\theta ^2)+ const$.
One can easily check that $v(r,\theta )$ satisfies the Laplace's equation,
$\Delta v=\frac{\partial ^2 v}{\partial r^2}+\frac{1}{r}\frac{\partial v}{\partial r}+
\frac{1}{r^2}\frac{\partial ^2 v}{\partial \theta ^2}
$, and its normal derivative vanishes on the unit circle.
\end{example}
\begin{example}
Harmonic function $u(x,y)=x^2-y^2$ on a unit circle equals to $\varphi =2 x^2 -1=\cos (2\theta)$.
To find the corresponding functions $u_1(z)$ and $u_2(\zeta )$, we make the substitution
$x=(z+\zeta )/2$ and $y=(z-\zeta )/(2i)$. This results in $u_1=z^2/2$ and $u_2=\zeta ^2/2$.
Thus the corresponding Dirichlet to Neumann operator is
$$
v(z,\zeta)=v(z_0,\zeta _0)
- \frac{1}{2}\int\limits _z ^{z_0} \frac{\tau ^2}{\tau}d\tau
-\frac{1}{2}\int\limits _\zeta ^{\zeta _0} \frac{\xi ^2}{\xi}d\xi .
$$
Thus, $v(x,y)=\frac{1}{2}(x^2-y^2) + const$, which is obviously a harmonic function. 
This function in polar coordinates has the form $v(r,\theta )= \frac{1}{2}r^2\cos (2\theta)$, whose normal derivative on the unit circle is $(\nabla v \cdot n)=\cos (2\theta)$, which equals to $\varphi$.
\end{example}
\begin{example}
Harmonic function $u(x,y)=x$ on a unit circle equals to $\varphi =\cos \theta$.
The corresponding functions  $u_1=z/2$ and $u_2=\zeta /2$.
Thus the expression for $v$ is
$$
v(z,\zeta)=v(z_0,\zeta _0)
- \frac{1}{2}\int\limits _z ^{z_0} \frac{\tau }{\tau}d\tau
-\frac{1}{2}\int\limits _\zeta ^{\zeta _0} \frac{\xi }{\xi}d\xi ,
$$
therefore, $v(x,y)=x + const$, which is obviously a harmonic function. 
This function in polar coordinates has the form $v(r,\theta )= \frac{1}{2}r^2\cos (2\theta)$, whose normal derivative on the unit circle equals to $\varphi$, $(\nabla v \cdot n)=\cos \theta$.
\end{example}

\section{Reflection about  an arc of a circle with nonhomogeneous Neumann data}
\label{R3}

Consider Neumann problem (\ref{N1}) in a neighborhood $\Omega$ of a point $(z_0,\bar z_0)\in \partial\mathbb{D}$. Let point $z\in \Omega _1\subset \Omega$ 
 for which formula (\ref{DN}) holds,
$$
v(z,\zeta)=v(z_0,\zeta _0)
- \int\limits _z ^{z_0} \frac{u_1(\tau )}{\tau}d\tau
-\int\limits _\zeta ^{\zeta _0} \frac{u_2(\xi )}{\xi}d\xi .
$$
For what follows, it is convenient to fix $\theta$ and to rewrite the integrals in polar coordinates,
\begin{equation}\label{DNp}
v(z,\zeta)=v(z_0,\zeta _0) - \int\limits _r ^{1} \frac{u(\rho e^{i\theta}, \rho e^{-i\theta})}{\rho}d\rho ,
\end{equation}
where $r=|z|$.

We remark that formula (\ref{E:0.1n}) for the unite circle can be rewritten as
\begin{equation} \label{RD}
 u(re^{i\theta},\frac{r}{e^{i\theta}}) + u(\frac{e^{i\theta}}{r},\frac{1}{re^{i\theta}})=\varphi (re^{i\theta},\frac{1}{re^{i\theta}})
+\varphi (\frac{e^{i\theta}}{r},\frac{r}{e^{i\theta}})
.
\end{equation}
Applying reflection, $R: \Omega \to \Omega$, to  formula (\ref{DNp}) and taking into account that the boundary points are fixed under $R$, we have
\begin{equation}\label{DNp1}
v(R(z,\zeta ))=v(z_0,\zeta _0)+I_1+I_2+I_3,
\end{equation}
where
$$
I_1= \int\limits _{\frac{1}{r}} ^{1} \frac{   u(\frac{e^{i\theta}}{\rho },\frac{1}{\rho e^{i\theta}})}{\rho}d\rho,
\quad
I_2=- \int\limits _{\frac{1}{r}} ^{1} \frac{ \varphi (\rho e^{i\theta}, \frac{1}{\rho e^{i\theta}})}{\rho}d\rho ,
\quad
I_3=- \int\limits _{\frac{1}{r}} ^{1} \frac{ \varphi (\frac{e^{i\theta}}{\rho}, \frac{\rho}{ e^{i\theta}})}{\rho}d\rho .
$$
To obtain a reflection formula for Neumann conditions, we need to rewrite the right hand side of (\ref{DNp1}) in terms of $v(z,\zeta)$. In order to do that, we make a substitution 
$\tilde \rho =1/\rho$ in $I_1$, which results in
$$
I_1= \int\limits _{\frac{1}{r}} ^{1} \frac{   u(\frac{e^{i\theta}}{\rho },\frac{1}{\rho e^{i\theta}})}{\rho}d\rho
=-\int\limits _{r} ^{1} \frac{   u(\tilde\rho{e^{i\theta}}, \tilde\rho{e^{-i\theta}}  )}{\tilde\rho}d\tilde\rho ,
$$
therefore,
$$
v(z_0,\zeta _0) +I_1=v(z,\zeta ).
$$
The second and third integrals could be combined if we do the same substitution,
$\tilde \rho =1/\rho$ in $I_3$. Indeed,
$$
I_3=- \int\limits _{1} ^{r} \frac{ \varphi (\tilde\rho{e^{i\theta}}, \frac{1}{\tilde\rho e^{i\theta}})}{\tilde\rho}d\tilde\rho , 
$$
hence,
$$
I_2+I_3=- \int\limits _{\frac{1}{r}} ^{r} \frac{ \varphi (\rho e^{i\theta}, \frac{1}{\rho e^{i\theta}})}{\rho}d\rho .
$$
Thus, we finally derived the following reflection formula for for harmonic functions satisfying Neumann condition on an arc  of the unite circle
\begin{equation}\label{RCircle}
v(\frac{e^{i\theta}}{r},\frac{1}{re^{i\theta}})=v(r{e^{i\theta}},r{e^{-i\theta}})- \int\limits _{\frac{1}{r}} ^{r} \frac{ \varphi (\rho e^{i\theta},\, \frac{1}{\rho e^{i\theta}})}{\rho}d\rho .
\end{equation}
\begin{example}
Consider a harmonic function, whose normal derivative on the unite circle, $x^2+y^2=1$, is constant, $\varphi =C$. Then formula
(\ref{RCircle}) reads as
\begin{equation}
v(\frac{e^{i\theta}}{r},\frac{1}{re^{i\theta}})=v(r{e^{i\theta}},r{e^{-i\theta}})-2C \ln r .
\end{equation}
\end{example}
\begin{example}
Consider a harmonic function, whose normal derivative on the unite circle, $x^2+y^2=1$, equals $\varphi =4x^2-2$. 
To compute the integral in formula (\ref{RCircle}), we rewrite function $\varphi$ as
$$
\varphi =4x^2-2=4(\frac{z+\zeta}{2})^2-2=z^2+\zeta ^2 +2(z\zeta -1).
$$
The last term vanishes on the complexification of the unit circle, therefore,
$\varphi (z,\zeta)= z^2 +\zeta ^2$, which implies 
$ \varphi (\rho e^{i\theta},\, \frac{1}{\rho e^{i\theta}})=\rho ^2 e^{2i\theta}+\frac{1}{\rho ^2 e^{2i\theta}}$. Thus,
$$
- \int\limits _{\frac{1}{r}} ^{r} \frac{ \varphi (\rho e^{i\theta},\, \frac{1}{\rho e^{i\theta}})}{\rho}d\rho 
=\int\limits _r ^{\frac{1}{r}}  \frac{\rho ^2e^{2i\theta}+ \frac{1}{\rho ^2e^{i\theta}}}{\rho}d\rho.
$$
Then formula
(\ref{RCircle}) reads as
\begin{equation}
v(\frac{e^{i\theta}}{r},\frac{1}{re^{i\theta}})=v(r{e^{i\theta}},r{e^{-i\theta}})+(\frac{1}{r^2}-r^2)\cos{2\theta} .
\end{equation}
\end{example}

\section{Neumann to Robin operator for an arc of the unit circle}
\label{R5}

Consider the following Robin problem 
\begin{equation}
\left\{
                \begin{array}{ll}\label{R1}
                  \Delta w=0 \qquad \mbox{near} \quad\partial\mathbb{D}\\
                  a\,w+\frac{\partial w}{\partial n}=\varphi _w          \qquad \mbox{on} \quad \partial\mathbb{D},\\
                 \end{array}
              \right.
\end{equation}
where $a$ and $b$ are real constants with $b\ne 0$.

Let $u$ and $v$ be harmonic functions defined near the arc 
$\partial\mathbb{D}$ of a unit circle in the plane:
\begin{equation}
\left\{
                \begin{array}{ll}\label{D1r}
                  \Delta u=0 \qquad \mbox{near} \quad\partial\mathbb{D}\\
                  u=\varphi  _u         \qquad \mbox{on} \quad \partial\mathbb{D},\\
                 \end{array}
              \right.
\end{equation}
\begin{equation}
\left\{
                \begin{array}{ll}\label{N1r}
                  \Delta v=0 \qquad \mbox{near} \quad\partial\mathbb{D}\\
                  \frac{\partial v}{\partial n}=\varphi  _v         \qquad \mbox{on} \quad \partial\mathbb{D}.\\
                 \end{array}
              \right.
\end{equation}
Here $\varphi_w$, $\varphi_u$, and $\varphi_v$ are  holomorphically continuable into $\mathbb{C}^2$ near $\partial\mathbb{D}$.

To derive a formula connecting functions $w(x,y)$, $u(x,y)$, and $v(x,y)$, we think of them as real parts of analytic functions, $w=\Re [h]$, $u=\Re [g]$, and $v=\Re [f]$ and assume that 
$\varphi_w=\varphi_u +\varphi_v$. 

For a point on a boundary, $z_0\in \partial\mathbb{D}$, we have the following equalities
$$
\frac{\partial v}{\partial n}(z_0)=D v\cdot n (z_0)=\Re [z_0 f^{\prime}(z_0)]=\varphi _v,
$$
$$
a\, w+\frac{\partial w}{\partial n}(z_0)=a\, w+D w\cdot n (z_0)=\Re [a\, h+b\, z_0 h^{\prime}(z_0)]=\varphi _w.
$$
Thus, setting
\begin{equation}\label{h}
a\, h(z)+b\, z\,h^{\prime}(z)=z\,f^{\prime}(z)+g(z)
\end{equation}
does not contradict the condition  $\varphi_w=\varphi_u +\varphi_v$. 
Multiplying (\ref{h}) by $z^{a/b-1}/b$, we have
$$
z^{\frac{a}{b}}h'(z)+\frac{a}{b}z^{\frac{a}{b}-1} h(z)=\frac{1}{b} z^{\frac{a}{b}} f'(z)+\frac{1}{b}z^{\frac{a}{b}-1} g(z), 
$$
which can be rewritten as
$$
(z^{\frac{a}{b}}h(z))'=\frac{1}{b} (z^{\frac{a}{b}} f(z))'+\frac{1}{b}z^{\frac{a}{b}-1} g(z)-\frac{a}{b^2}z^{\frac{a}{b}-1}f(z).
$$
Integrating from $z$ to $z_0$ along a segment with constant argument, we obtain
$$
z_0^{\frac{a}{b}}h(z_0)-z^{\frac{a}{b}}h(z)=
\frac{1}{b} (z_0^{\frac{a}{b}} f(z_0)-z^{\frac{a}{b}} f(z))
+\frac{1}{b}\int\limits _z^{z_0} \zeta ^{\frac{a}{b}-1}[g(\zeta )-\frac{a}{b}
f(\zeta )]\, d\zeta ,
$$
which implies
$$
h(z)=\Bigl (\frac{z_0}{z}\Bigr )^{\frac{a}{b}}\Bigl (h(z_0)-\frac{1}{b}f(z_0)\Bigr )
+\frac{1}{b}f(z)
-\frac{1}{b}z^{-\frac{a}{b}}\int\limits _z^{z_0} \zeta ^{\frac{a}{b}-1}[g(\zeta )-\frac{a}{b}
f(\zeta )]\, d\zeta .
$$
Taking the real part from both sides of the latter formula, one can express a solution of the Robin problem
(\ref{R1}) in terms of the corresponding solutions to  the Dirichlet and Neumann problems (\ref{D1r}) and (\ref{N1r}).
Alternatively, one can divide formula (\ref{h}) by z and then integrate, which results 
\begin{equation}\label{hf}
f(z)=f(z_0)+b(h(z)-h(z_0))-
a\, \int\limits _z^{z_0} \frac {h(\zeta)}{\zeta}d\zeta
-\int\limits _z^{z_0} \frac {g(\zeta)}{\zeta}d\zeta .
\end{equation}
Assuming  $\varphi _u=\varphi _v$ and taking into account formula (\ref{DN}), we obtain a Robin to Neumann operator
$$
v(z,\zeta)=v(z_0,\zeta _0)+\frac{b}{2}[w(z,\zeta)-w(z_0,\zeta _0)]-
\frac{a}{2}\, \int\limits _z^{z_0} \frac {w_1(\tau)}{\tau}d\tau
-\frac{a}{2}\, \int\limits _{\zeta}^{\zeta _0} \frac {w_2(\xi)}{\xi}d\xi,
$$
where $w(z,\zeta )=w_1(z) +w_2(\zeta )$.
\begin{theorem}
Let $\partial\mathbb{D}$ be the unit circle in the plane centered at the origin. 
If $w(z,\zeta)$ is a solution to (\ref{R1}), then a Robin to Neumann operator has the form
\begin{equation}\label{RN}
v(z,\zeta)=v(z_0,\zeta _0)+\frac{b}{2}[w(z,\zeta)-w(z_0,\zeta _0)]-
\frac{a}{2}\, \int\limits _z^{z_0} \frac {w_1(\tau)}{\tau}d\tau
-\frac{a}{2}\, \int\limits _{\zeta}^{\zeta _0} \frac {w_2(\xi)}{\xi}d\xi .
\end{equation}
Here $v(z,\zeta)$ is a harmonic function such that 
$\frac{\partial v}{\partial n}=\varphi _v=\varphi _w/2$ on $\partial\mathbb{D}$.
\end{theorem}
\begin{proof}
To prove  formula (\ref{RN}) one  needs to check  the conditions in (\ref{N1r}).
Obviously, function $v(z,\zeta ) $ is harmonic, $\frac{\partial ^2 v}{\partial z\partial \zeta}=0$ if $w(z,\zeta )$ is harmonic.

To check the second condition we differentiate (\ref{RN}) with respect to $r$,
$$
\frac{\partial}{\partial r}v(z,\zeta)=\frac{b}{2}\frac{\partial}{\partial r}w(z,\zeta )
+\frac{a}{2}\frac {w_1(z)}{z}e^{i\theta}+\frac{a}{2}\frac {w_2(\zeta )}{z}e^{-i\theta} .
$$
Then on $\partial\mathbb{D}$, the latter formula reduces to
$$
\frac{\partial}{\partial r}v(z_0,\zeta _0)=\frac{b}{2}\frac{\partial}{\partial r}w(z_0,\zeta _0)
+\frac{a}{2}\frac {w_1(z_0)}{e^{i\theta}}e^{i\theta}+\frac{a}{2}\frac {w_2(\zeta _0 )}{e^{-i\theta}}e^{-i\theta},
$$
resulting in
$$
\frac{\partial}{\partial r}v(z_0,\zeta _0)=\varphi _v =
\frac{b}{2}\frac{\partial}{\partial r}w(z_0,\zeta _0)+\frac{a}{2}{w(z_0,\zeta _0)}=
\frac{1}{2}\varphi _w .
$$
This finishes the proof.
\end{proof}
\begin{example}
Consider a solution to Robin problem $w=\ln r$ satisfying the Robin condition with 
$\varphi _w=b$
on $\partial\mathbb{D}$.
Function $w$ has the representation 
$w(z,\zeta )=w_1 (z) +w_2 (\zeta )=\frac{1}{2}\ln z+ \frac{1}{2}\ln \zeta $.
Then formula (\ref{RN}) reads as
$$
v(z,\zeta)=v(z_0,\zeta _0)+\frac{b}{2}\ln r -
\frac{a}{4}\, \int\limits _z^{z_0} \frac {\ln(\tau)}{\tau}d\tau
-\frac{a}{4}\, \int\limits _{\zeta}^{\zeta _0} \frac {\ln (\xi)}{\xi}d\xi ,
$$
therefore,
$$
v(z,\zeta)=\frac{b}{2}\ln r+\frac{a}{4}(\ln^2 r-\theta ^2)+ const,
$$ 
which is a harmonic function satisfying the Neumann condition on $\partial\mathbb{D}$ with $\varphi _v =\frac{b}{2}$.
\end{example}
\begin{corollary} 
Solutions to the Dirichlet problem (\ref{D1r}) are related to the solutions to the Robin problem (\ref{R1}),
\begin{equation}\label {DR}
u=\frac{a}{2}w+\frac{b\,r}{2}\frac{\partial w}{\partial r}
\end{equation}
whenever $\varphi _u=\frac{1}{2}\varphi _w$ on $\partial\mathbb{D}$.
\end{corollary}
Indeed,  formulae (\ref{DN}) and (\ref{RN}) imply
$$
\frac{b}{2}[w(z,\zeta)-w(z_0,\zeta _0)]-
\frac{a}{2}\, \int\limits _z^{z_0} \frac {w_1(\tau)}{\tau}d\tau
-\frac{a}{2}\, \int\limits _{\zeta}^{\zeta _0} \frac {w_2(\xi)}{\xi}d\xi
= - \int\limits _z ^{z_0} \frac{u_1(\tau )}{\tau}d\tau
-\int\limits _\zeta ^{\zeta _0} \frac{u_2(\xi )}{\xi}d\xi .
$$
Assuming that $(x,y)\in \mathbb{R}^2$ and setting $z=re^{i\theta}$, $\zeta =re^{-i\theta}$, we have
$$
\frac{b}{2}[w(re^{i\theta},re^{-i\theta})-w(z_0,\zeta _0)]-
\frac{a}{2}\, \int\limits _r^{1} \frac {w(\rho e^{i\theta},\rho e^{-i\theta})}{\rho}d\rho
= -\int\limits _r^{1} \frac {u(\rho e^{i\theta},\rho e^{-i\theta})}{\rho}d\rho .
$$
Differentiation of the latter formula with respect to $r$  results in (\ref{DR}).


\section{Reflection about a circular arc with nonhomogeneous Robin data}
\label{R6}

To derive a reflection formula for solutions to (\ref{R1}) we use formula 
(\ref{RCircle}),
$$
v(\frac{e^{i\theta}}{r},\frac{1}{re^{i\theta}})=v(r{e^{i\theta}},r{e^{-i\theta}})- \int\limits _{\frac{1}{r}} ^{r} \frac{ \varphi _v (\rho e^{i\theta},\, \frac{1}{\rho e^{i\theta}})}{\rho}d\rho 
$$
and formula (\ref{RN}),
$$
v(z,\zeta)=v(z_0,\zeta _0)+\frac{b}{2}[w(z,\zeta)-w(z_0,\zeta _0)]-
\frac{a}{2}\, \int\limits _z^{z_0} \frac {w_1(\tau)}{\tau}d\tau
-\frac{a}{2}\, \int\limits _{\zeta}^{\zeta _0} \frac {w_2(\xi)}{\xi}d\xi ,
$$
which in the plane can be rewritten in variables $(r, \theta )$ as
$$
v(re^{i\theta},re^{-i\theta})=\frac{b}{2}w(re^{i\theta},re^{-i\theta})-
\frac{a}{2}\, \int\limits _r^{1} \frac {w(\rho e^{i\theta},\rho e^{-i\theta})}{\rho}d\rho +
 \widetilde C,
$$
where $\widetilde C=v(z_0,\zeta _0)-\frac{b}{2}w(z_0,\zeta _0)$. This implies,
$$
v(\frac{e^{i\theta}}{r},\frac{1}{re^{i\theta}})=
\frac{b}{2}w(\frac{e^{i\theta}}{r},\frac{1}{re^{i\theta}})-
\frac{a}{2}\, \int\limits _{1/r}^{1} \frac {w(\rho e^{i\theta},\rho e^{-i\theta})}{\rho}d\rho +
 \widetilde C,
$$
Making a substitution $\widetilde\rho=1/\rho$ in the integral, we have
$$
v(\frac{e^{i\theta}}{r},\frac{1}{re^{i\theta}})=
\frac{b}{2}w(\frac{e^{i\theta}}{r},\frac{1}{re^{i\theta}})+
\frac{a}{2}\, \int\limits _{r}^{1} \frac {w(\frac{1}{\widetilde\rho} e^{i\theta},\frac{1}{\widetilde\rho} e^{-i\theta})}{\widetilde\rho}d\widetilde\rho +
 \widetilde C.
$$
Plugging into formula (\ref{RCircle}), we obtain
$$
\frac{b}{2}w(\frac{e^{i\theta}}{r},\frac{1}{re^{i\theta}})+
\frac{a}{2}\, \int\limits _{r}^{1} \frac {w(\frac{1}{\widetilde\rho} e^{i\theta},\frac{1}{\widetilde\rho} e^{-i\theta})}{\widetilde\rho}d\widetilde\rho 
 =\frac{b}{2}w(re^{i\theta},re^{-i\theta})
$$
$$
-
\frac{a}{2}\, \int\limits _r^{1} \frac {w(\rho e^{i\theta},\rho e^{-i\theta})}{\rho}d\rho 
-\int\limits _{\frac{1}{r}} ^{r} \frac{ \varphi _v (\rho e^{i\theta},\, \frac{1}{\rho e^{i\theta}})}{\rho}d\rho ,
$$
which leads to the following theorem.
\begin{theorem}
Let $w$ be a solution to the problem (\ref{R1}), then for a pair of points in the plain symmetric with respect to the unit circle centered at the origin, the following reflection formula holds
\begin{equation}\label{RR}
w(\frac{e^{i\theta}}{r}) 
 =w(re^{i\theta})-
\frac{a}{b}\, \int\limits _r^{1} \frac {w(\rho e^{i\theta})+w(\frac{1}{\rho} e^{i\theta})}{\rho}d\rho 
-\frac{1}{b}\int\limits _{\frac{1}{r}} ^{r} \frac{ \varphi _w (\rho e^{i\theta},\, \frac{1}{\rho e^{i\theta}})}{\rho}d\rho .
\end{equation}
\end{theorem}
\begin{example}
Let $w$ be a solution to the problem (\ref{R1}) with $\varphi _w=b$.
Then formula (\ref{RR}) reads as
$$
w(\frac{e^{i\theta}}{r}) 
 =w(re^{i\theta})-
\frac{a}{b}\, \int\limits _r^{1} \frac {w(\rho e^{i\theta})+w(\frac{1}{\rho} e^{i\theta})}{\rho}d\rho 
-2\ln r .
$$
\end{example}
\begin{example}
Let $w$ be a solution to the problem (\ref{R1}) with $\varphi _w=(a+2b)(2x^2-1)$. 
Taking into account that $\varphi _w=\frac{a+2b}{2}(z^2+\zeta ^2)$ on $\Gamma _{\mathbb{C}}$,
 formula (\ref{RR}) reduces to
$$
w(\frac{e^{i\theta}}{r}) 
 =w(re^{i\theta})-
\frac{a}{b}\, \int\limits _r^{1} \frac {w(\rho e^{i\theta})+w(\frac{1}{\rho} e^{i\theta})}{\rho}d\rho 
-\frac{a+2b}{2b}\int\limits _{\frac{1}{r}} ^{r} \frac{ \rho ^2 e^{2i\theta}+\frac{1}{\rho ^2 e^{2i\theta}})}{\rho}d\rho ,
$$
which results in
$$
w(\frac{e^{i\theta}}{r}) 
 =w(re^{i\theta})-
\frac{a}{b}\, \int\limits _r^{1} \frac {w(\rho e^{i\theta})+w(\frac{1}{\rho} e^{i\theta})}{\rho}d\rho 
-\frac{a+2b}{2b}(r^2-\frac{1}{r^2}) \cos (2\theta) .
$$
\end{example}
\begin{example}
Let $w$ be a solution to the problem (\ref{R1}) with $\varphi _w=(a+b)\cos\theta$.
Then the corresponding reflection formula is
$$
w(\frac{e^{i\theta}}{r}) 
 =w(re^{i\theta})-
\frac{a}{b}\, \int\limits _r^{1} \frac {w(\rho e^{i\theta})+w(\frac{1}{\rho} e^{i\theta})}{\rho}d\rho 
-\frac{a+b}{2b}(r-\frac{1}{r}) \cos \theta .
$$
\end{example}

\section{Dirichlet to Neumann operator and reflection about an arc of an algebraic curve}
\label{R7}

In this section we generalize the results obtained for a circular arc in sections \ref{R2} and \ref{R3} for the case of an algebraic curve $\Gamma$.
First, we discuss a Dirichlet to Neumann mapping. Then we apply this mapping for derivation of a reflection formula. Remark that this formula was obtained at \cite{S99}, \cite{S18} using different methods.

Consider the following two problems for functions $u(x,y)$ and $v(x,y)$ defined near an arc 
 of an algebraic curve $\Gamma$ in the plane:
\begin{equation}
\left\{
                \begin{array}{ll}\label{Dg}
                  \Delta u=0 \qquad \mbox{near} \quad\Gamma\\
                  u=\varphi           \qquad \mbox{on} \quad \Gamma,\\
                 \end{array}
              \right.
\end{equation}
\begin{equation}
\left\{
                \begin{array}{ll}\label{Ng}
                  \Delta v=0 \qquad \mbox{near} \quad\Gamma\\
                  \frac{\partial v}{\partial n}=\varphi           \qquad \mbox{on} \quad \Gamma,\\
                 \end{array}
              \right.
\end{equation}
where $\varphi$ is  holomorphically continuable into $\mathbb{C}^2$ near $\Gamma$.
\begin{theorem}
Let $\Gamma$ be an arc of an algebraic curve in the plane. 
If $u(z,\zeta)$ is a solution to (\ref{Dg}), then 
\begin{equation}\label{DNg}
v(z,\zeta)=v(z_0,\zeta _0)
+ i \int\limits _z ^{z_0} u_1(\tau )\sqrt{S^{\prime}(\tau )}d\tau
- i\int\limits _\zeta ^{\zeta _0} u_2(\xi )\sqrt{\widetilde S^{\prime}(\xi )}d\xi 
\end{equation}
is a solution to (\ref{Ng}). 
Here $u(z,\zeta)=u_1(z)+u_2(\zeta )$ is a harmonic function such that 
$u(z,\bar z )=\varphi$ on $\Gamma$.
\end{theorem}

\begin{proof}
According to formula (\ref{DNg}), function $v(z,\zeta ) $ has a representation 
as a sum of a function $v_1(z)$ and a function $v_2(\zeta )$, and therefore, is harmonic, 
$\frac{\partial ^2 v}{\partial z\partial \zeta}=0$.

To check the  second condition in (\ref{Ng}), note that the normal derivative could be computed using the formula \cite{davis},
$$
\frac{\partial v}{\partial n}(z,\zeta )=\frac{i}{\sqrt{S^{\prime}(z )}}
\Bigl (\frac{\partial v_1}{\partial z}-\frac{\partial v_2}{\partial \zeta}S^{\prime}(z)
\Bigr )=\varphi .
$$
Differentiating formula (\ref{DNg}), we have
$$
\frac{\partial v}{\partial n}(z,\zeta )=\frac{i}{\sqrt{S^{\prime}(z )}}
\Bigl (-iu_1(z )\sqrt{S^{\prime}(z )}-iu_2(\zeta )\sqrt{\widetilde S^{\prime}(\zeta )}
 S^{\prime}(z)  \Bigr ),
$$
therefore, on the  curve $\Gamma$ we obtain
$$
\frac{\partial v}{\partial n}(z,\zeta )_{|_{\Gamma}}=
\Bigl (u_1(z )+u_2(\zeta )\Bigr )_{|_{\Gamma}}=u(z,\zeta )_{|_{\Gamma}}=\varphi .
$$
This finishes the proof.
\end{proof}

\begin{remark}
For the special case when $\Gamma$ is the unit circle centered at the origin, $S(z)=1/z$, formula (\ref{DNg}) reduces to formula (\ref{DN}). 
\end{remark}

To derive the reflection formula generalizing formula (\ref{RCircle}),
let us rewrite formula (\ref{DNg}) at the reflected point,
\begin{equation}\label{DNgr}
v(\widetilde S(\zeta ),S(z))=v(z_0,\zeta _0)
+ i \int\limits _{\widetilde S(\zeta )} ^{\widetilde S(\zeta _0)} u_1(\tau )\sqrt{S^{\prime}(\tau )}d\tau
- i\int\limits _{S(z)} ^{S(z _0)} u_2(\xi )\sqrt{\widetilde S^{\prime}(\xi )}d\xi , 
\end{equation}
and subtract (\ref{DNgr}) from (\ref{DNg}),
$$
v(z,\zeta) -v(\widetilde S(\zeta ),S(z))=
i \int\limits _{z} ^{z_0} u_1(\tau )\sqrt{S^{\prime}(\tau )}d\tau
+i\int\limits _{S(z)} ^{S(z _0)} u_2(\xi )\sqrt{\widetilde S^{\prime}(\xi )}d\xi
$$
$$
-i\int\limits _\zeta ^{\zeta _0} u_2(\xi )\sqrt{\widetilde S^{\prime}(\xi )}d\xi
- i \int\limits _{\widetilde S(\zeta )} ^{\widetilde S(\zeta _0)} u_1(\tau )\sqrt{S^{\prime}(\tau )}d\tau . 
$$
Making the substitution $\xi =S(\tau)$ in the second and third integrals
$$
i\int\limits _{S(z)} ^{S(z _0)} u_2(\xi )\sqrt{\widetilde S^{\prime}(\xi )}d\xi
=i \int\limits _{z} ^{z_0} u_2(S(\tau ) )\sqrt{S^{\prime}(\tau )}d\tau ,
$$
$$
-i\int\limits _\zeta ^{\zeta _0} u_2(\xi )\sqrt{\widetilde S^{\prime}(\xi )}d\xi
=- i \int\limits _{\widetilde S(\zeta )} ^{\widetilde S(\zeta _0)} u_2(S(\tau ))\sqrt{S^{\prime}(\tau )}d\tau ,
$$
we have
$$
v(z,\zeta) -v(\widetilde S(\zeta ),S(z))=
i \int\limits _{z} ^{z_0}[ u_1(\tau )+u_2((S(\tau))]\sqrt{S^{\prime}(\tau )}d\tau
$$
$$
- i \int\limits _{\widetilde S(\zeta )} ^{\widetilde S(\zeta _0)} 
[u_1(\tau )+u_2(S(\tau))]\sqrt{S^{\prime}(\tau )}d\tau .
$$
Taking into account that $u_1(z )+u_2(S(z))=u(z,S(z))=\varphi (z,S(z))$ and $z_0=\widetilde S(\zeta_0)$, we arrive at
the formula generalizing (\ref{RCircle}), 
\begin{equation}
 v(\widetilde S(\zeta ),S(z))=v(z,\zeta)+i\int\limits _{\widetilde S(\zeta )} ^{z} 
\varphi(\tau ,S(\tau))\sqrt{S^{\prime}(\tau )}d\tau .
\end{equation}




\end{document}